\documentclass[12pt,tikz]{article}
\usepackage{amsmath, amsthm, amssymb,enumitem}
\usepackage[top=1cm,left=1cm,right=1cm,bottom=1.5cm]{geometry}

\theoremstyle{plain}
\usepackage[table]{xcolor}
\usepackage{tikz}
\usetikzlibrary{patterns}
\usetikzlibrary{calc}

\usepackage{diagbox}

\newtheorem{theorem}{Theorem}
\newtheorem{lemma}[theorem]{Lemma}

\newtheorem{corollary}[theorem]{Corollary}

\usepackage{array,color,colortbl}
\providecommand{\mk}{\cellcolor[gray]{.8}}

\def\cref#1{Conjecture~$\ref{#1}$}
\def\Cref#1{Corollary~$\ref{#1}$}

\renewcommand{\geq}{\geqslant}
\renewcommand{\leq}{\leqslant}

\renewcommand{\emptyset}{\varnothing}

\def\={\equiv}

\def\dfrac#1#2{\lower0.15ex\hbox{\large$\frac{#1}{#2}$}} 

\newcommand{\Mod}[1]{\ (\mathrm{mod}\ #1)}

\title{Defining sets which intersect each Latin trade at least twice}

\author{
	Richard Bean\thanks{
		Cyber Research Centre, University of Queensland, 4072, Australia.}\\ 
	\texttt{r.bean1@uq.edu.au}\\
	\and
	Nicholas J. Cavenagh\thanks{
		Department of Mathematics,
		The University of Waikato,
		Private Bag 3105,
		Hamilton 3240, New Zealand.}\\
	\texttt{nickc@waikato.ac.nz}\\ 
}

\begin{document}
	
	\date{}
	\maketitle
	
	\begin{abstract}
		A defining set of a Latin square is a partially filled-in Latin square which completes to no other Latin square of the same order. We introduce the concept of a $k$-strong defining set, in which if less than $k$ entries are deleted, the property of being a defining set is retained. 
		Equivalently, a $k$-strong defining set intersects every Latin trade in the Latin square at least $k$ times. 
		In the addition table for integers modulo $n$, when $n$ is even we determine the minimum size of a $k$-strong defining set for any $k$. For odd $n$ we give a construction for a minimally $2$-strong defining set. We furthermore give computational results for Latin squares of small orders. 
	\end{abstract}
	
	
	
	\section{Introduction}
	
	In what follows, rows and columns of an $n\times n$ array $L$ are each
	indexed by a set  $N(n)$ of size $n$, with $L_{i,j}$ denoting the \emph{symbol} 
	in cell $(i,j)$. We sometimes consider an array~$L$ to be a
	set of ordered triples $L=\{(i,j;L_{i,j})\}$ so that the notion of a
	subset of an array is precise. 
	A \emph{partial Latin square} (PLS) of \emph{order} $n$ is an $n\times
	n$ array in which some cells may be empty, the set of possible symbols 
	$N(n)$ has size $n$, and
	each symbol occurs at most once in each row and
	column.
	A \emph{Latin square} is a PLS with no empty cells.  
	
	A \emph{Latin trade} is a non-empty PLS $T$ such that there exists another PLS $T'$ of the same order, where: (a) $T\cap T'=\emptyset$, (b) 
	$T$ and $T'$ have the same set of non-empty cells, (c) $T$ and $T'$ share the same set of symbols in each row; and 
	(d) $T$ and $T'$ share the same set of entries in each column. We say that $T'$ is the \emph{disjoint mate} of the Latin trade $T$. In some literature the pair $(T,T')$ are together called a \emph{bitrade}. 
	If we take any two distinct Latin squares $L$ and $L'$ of the same order, observe that 
	$L\setminus L'$ is a Latin trade with disjoint mate  $L'\setminus L$.  Thus Latin trades describe the differences between a Latin square and any other square  of the same order.

	A \emph{defining set} $D$ of a Latin square $L$ is a subset of $L$ such that if $L'$ is a distinct Latin square of the same order as $L$, $D$ is not a subset of $L'$. We sometimes say that $D$ has a \emph{unique completion} to $L$. If $D$ is minimal with respect to this property we say that $D$ is a \emph{critical set} of $L$.
	Necessarily, a defining set $D$ of a Latin square $L$ intersects every Latin trade $T$ in $L$. 
	(If not, $D$ completes to the distinct Latin squares $L$ and $(L\setminus T)\cup T'$.) 
	Thus the theory of defining sets and Latin trades are very much intertwined.

	We define $B_n$ to be the operation table for the integers modulo $n$. 
	That is, 
	$$B_n:=\{(i,j;i+j\Mod{n}): i,j\in {\mathbb Z}_n \}.$$
	This particular Latin square may have some extremal properties in terms of Latin trades and critical sets, which we now outline. 
	Let $p$ be the smallest prime dividing $n$. 
	The size of the smallest Latin trade in $B_n$ was shown to be at least $e\log{p}+3$  \cite{DK}. 
	Conversely, in \cite{Sz} it is shown that for all $n$, $B_n$ contains a Latin trade of 
	size at most $5\log_2{p}$. 
	The smallest possible size of a Latin trade is $4$, given by any $2\times 2$ subsquare. Such Latin trades are called {\em intercalates}. 
	Meanwhile, any Latin square has a Latin trade of size at most $8\sqrt{n}$ \cite{CR}. 
	It is conjectured in \cite{CR} that if $L$ is a Latin square of order prime $p$, the smallest 
	Latin trade in $L$ is no larger than the smallest Latin trade in $B_p$. 
	
	Let scs$(L)$ be the size of the smallest critical set in a Latin square $L$ and let scs$(n)$ in turn be the size of the smallest critical set among all Latin squares of order $n$. 
	In \cite{HV} it was shown that 
	scs$(n)\geq n^2/10^4$ for sufficiently large $n$. 
	Conversely, it has been shown that scs$(n)\leq \lfloor n^2/4\rfloor$ for all $n\geq 1$ \cite{CDS,DC}. 
	Moreover, it is conjectured that equality holds in this bound, a conjecture that we know is true when $n\leq 8$ \cite{Be05}. 
	Interestingly, examples of critical sets of size $\lfloor n^2/4\rfloor$ are, in the present state of the literature,  known only to exist in $B_n$ \cite{CDS,DC}. 
	
	We next introduce an idea which measures the amount of information in a defining set that can be ``lost'' while retaining the property of being a defining set. 
	We say that a defining set $D$ of a Latin square $L$ is  \emph{$k$-strong} 
	if for every trade $T\subseteq L$, 
	$|D\cap T|\geq k$. 
	In turn, a defining set $D$ is said to be \emph{minimally $k$-strong} if it is $k$-strong but any strict subset of $D$ is not.  
	By definition, any defining set is 
	$1$-strong and any critical set is minimally $1$-strong. 
	Observe also that if $D$ is a $k$ strong defining set, than $D$ is also a $k'$-strong defining set for each $1\leq k'<k$. 
	
	Since the number of trades in a Latin square of order $n$ is one less than the number of Latin squares of order $n$, verifying that a subset is $k$-strong by checking every trade can be cumbersome. The following lemma gives a more efficient method, which we make use of in Section 3.

	\begin{lemma}
		Let $D\subset L$ where $L$ is a Latin square.
		Then $D$ is a $k$-strong defining set if and only if for any $D'\subset D$ such that $|D'|< k$, $D\setminus D'$ is a defining set. 
		A $k$-strong defining set $D$ is, in turn, minimally $k$-strong if and only if for each 
		triple $(i,j;L_{i,j})\in D$ 
		there exists a Latin trade $T\subseteq L$ such that
		$(i,j;L_{i,j})\in T$ and 
		$|T\cap D|=k$. 
		\label{intermed}
	\end{lemma}
	
	\begin{proof}
		Suppose first that $D$ is $k$-strong. 
		Then by definition, $D$ intersects every Latin trade in $L$ at least $k$ times. In turn, for any $D'\subseteq D$
		such that $|D'|<k$, 
		$D\setminus D'$ intersects every Latin trade in $L$ at least once, and is thus a defining set. 
		Conversely, suppose that 
		$D$ is a subset of a Latin square 
		such that for any $D'\subset D$ such that $|D'|< k$, $D\setminus D'$ is a defining set. 
		Suppose there exists a trade $T\subseteq L$ such that 
		$|T\cap D|<k$. 
		Let $T'$ be a disjoint mate of $T$. 
		Then letting $D'=T\cap D$, 
		$D\setminus D'$ is a subset of the Latin square 
		$(L\setminus T)\cup T'\neq L$. Therefore $D\setminus D'$ is not a defining set for $L$, a contradiction.  
		Thus the first claim of the lemma is satisfied. 
		
		Next, suppose that 
		$D$ is minimally $k$-strong. Then for 
		any $(i,j;L_{i,j})\in D$, 
		$D'=D\setminus \{(i,j;L_{i,j})\}$ is not $k$-strong.
		Thus there exists a Latin trade $T\subseteq L$ such that   
		$|T\cap D|\geq k$ but  
		$|T\cap D'|<k$. Thus $\{(i,j;L_{i,j})\}\in T$. 
		Finally, 
		suppose that $D$ is $k$-strong but not minimally so.
		Then there exists  $(i,j;L_{i,j})\in D$
		such that $D\setminus \{(i,j;L_{i,j})\}$ is $k$-strong.
		By definition, any trade $T$ intersects 
		$D\setminus \{(i,j;L_{i,j})\}$
		at least $k$ times. 
		Therefore for any trade $T$ 
		such that  $\{(i,j;L_{i,j})\}\in T$, 
		$|T\cap D|>k$.
	\end{proof}
	
	Observe the following. 
	
	\begin{lemma}
		A Latin square $L$ is $d$-strong if and only if $d$ is the size of the smallest Latin trade in $L$. 
		\label{minny}
	\end{lemma}

	Note that $L$ is only in turn minimally $d$-strong if every element of $L$ belongs to a Latin trade of size $d$. For example, there is a Latin square of order $5$ which contains an intercalate but for which not every element is in an intercalate. 
	Since Latin squares without intercalates are rare as $n$ grows large  
	(see \cite{MW}), most Latin squares are $4$-strong.  
	The following implies that if $d$ is the size of the smallest Latin trade in a Latin square $L$, we can find
	a chain of PLSs $D_1\subset D_2\subset \dots \subset D_d\subseteq L$ such that $D_k$ is minimally $k$-strong for 
	$1\leq k\leq d$. 
	\begin{lemma}
		Let $P$ be a minimally $k$-strong subset of a Latin square $L$ of order $n\geq 2$, where $k\geq 1$. 
		Then there is a {\rm PLS} $Q\subset P$ such that 
		$Q$ is minimally $(k-1)$-strong.  
		\label{CHAIN}
	\end{lemma}
	
	\begin{proof}
		If $k=1$, then $Q=\emptyset$, since every element of a Latin square belongs to some trade (such as that formed by swapping two rows). 	
		
		Otherwise, for any $(i,j;k)\in P$, by definition, 
		$Q'=P\setminus \{(i,j;k)\}$ is not $k$-strong.
		On the other hand, every trade in $L$ intersects $Q'$ at least $k-1$ times. 
		For each element of $Q'$, recursively remove that element if and only if the remaining PLS is $(k-1)$-strong. 
		The final result of this algorithm is a minimally
		$(k-1)$-strong  PLS
		$Q\subseteq Q'\subset P$.   	
	\end{proof}

	For any $S\subseteq B_n$ and integers $a$ and $b$, we define 
	$$S\oplus (a,b):=
	\{(i+a\Mod{n}, 
	j+b\Mod{n}; k+a+b\Mod{n}): 
	(i,j;k)\in S 
	)\}.$$
	Observe that 
	$S\oplus (a,b)\subseteq B_n$. 
	Let $T_n$ be a Latin trade in $B_n$ of minimum size $d_n$. 
	Since $T_n\oplus (a,b)$ is also a Latin trade in $B_n$ for any integers $a$ and $b$, it follows from Lemma \ref{minny} that the Latin square $B_n$ is itself a minimal 
	$d_n$-strong defining set.   
	
	The main results of this paper are as follows. 
	Let sds$(L,k)$ be the size of the smallest $k$-strong defining set in the Latin square $L$. 
	In Section 2, we prove the following. 
	\begin{theorem}
		If $n$ is even, 
		{\rm sds}$(B_n,k)=kn^2/4$ where $1\leq k\leq 4$. 
		\label{themaintwo}
	\end{theorem}

	Next, for each $i\in {\mathbb Z}_n$, we define the {\em diagonal}  $D_i\subset B_n$ to be
	the PLS 
	$$\{(r,r+i;2r+i\Mod{n}): 
	r\in  {\mathbb Z}_n\}.$$
	In Section 3, we prove the following. 
	\begin{theorem}
		The {\rm PLS} given by 
		$P_n:=D_0\cup D_1\cup\dots \cup 
		D_{\lceil (n-3)/2\rceil}$ is
		a minimally $2$-strong defining set of $B_n$ for all $n\geq 2$.  
		\label{themainone} 
	\end{theorem}
	See Figure \ref{firstegg} for an example of $P_{11}$. 
	In Section 3 we give computational results for small orders. Finally in Section 4, we describe 
	$k$-strong defining sets in $B_n$ for larger values of $k$.
	
	\begin{figure}
		$$\begin{array}{|c|c|c|c|c|c|c|c|c|c|c|}
			\hline
			\mk 0 & \mk 1 & \mk 2 & \mk 3 & \mk 4 & 5 & 
			6 & 7 & 8 & 9 & 10 \\
			\hline
			1 & \mk 2 & \mk 3 & \mk 4 & \mk 5 & 
			\mk 6 & 7 & 8 & 9 & 10 & 0 \\
			\hline
			2 & 3 & \mk 4 & \mk 5 & 
			\mk 6 & \mk 7 & \mk 8 & 9 & 10 & 0 & 1 \\
			\hline
			3 & 4 & 5 & 
			\mk 6 & \mk 7 & \mk 8 & \mk 9 & \mk 10 & 0 & 1 & 2\\
			\hline
			4 & 5 & 
			6 & 7 & \mk 8 & \mk 9 & \mk 10 & \mk 0 & \mk 1 & 2 & 3 \\
			\hline
			5 & 
			6 & 7 & 8 & 9 & \mk 10 & \mk 0 & \mk 1 & \mk 2 & \mk 3 & 4 \\
			\hline
			6 & 7 & 8 & 9 & 10 & 0 & \mk 1 & \mk 2 & \mk 3 & \mk 4 & \mk 5  \\
			\hline
			\mk 7 & 8 & 9 & 10 & 0 & 1 & 2 & \mk 3 & \mk 4 & \mk 5 & \mk 6  \\
			\hline
			\mk 8 & \mk 9 & 10 & 0 & 1 & 2 & 3 & 4 & \mk 5 & \mk 6 & \mk 7 \\
			\hline
			\mk 9 & \mk 10 & \mk  0 & 1 & 2 & 3 & 4 & 5 & 6 & \mk 7 & \mk 8 \\
			\hline
			\mk 10 & \mk  0 & \mk 1 & \mk 2 & 3 & 4 & 5 & 6 &  7 &  8 & \mk 9 \\
			\hline
		\end{array}$$
		\caption{The PLS $P_{11}$}
		\label{firstegg}
	\end{figure}

	\section{Disjoint critical sets in Latin squares and the even case}
	
	In \cite{ABK}, the question of whether a Latin square can be partitioned into disjoint critical sets is studied. 
	This is of interest to our problem, because:
	\begin{lemma}
		If $C_1,\dots, 
		C_m$ are pairwise disjoint defining sets in a Latin square $L$, then  
		$C_1\cup C_2 \cup \dots \cup  C_m$
		is an $m$-strong defining set in $L$.
		\label{DISJ}
	\end{lemma}
	
	\begin{proof}
		Let $C_1,\dots, 
		C_m$ be pairwise disjoint defining sets in a Latin square $L$. 
		As observed in the Introduction, any trade $T$ in $L$ must intersect each of the defining sets. Thus, $C_1\cup C_2 \cup \dots \cup  C_m$
		is an $m$-strong defining set in $L$.
	\end{proof}
	
	In \cite{ABK}, it is shown, in particular, that $B_n$ can be partitioned into $4$ disjoint critical sets for any integer $n$. 
	For $n$ even, this result allows us to determine sds$(B_n,k)$ exactly, as we now show.

	\begin{theorem} {\rm \cite{ABK}} 
		For any even $n\geq 2$, $B_n$ partitions into the following $4$ critical sets, each of size $n^2/4$: 
		$$C_1:= \{(i,j;i+j) : 0\leq i,j;\ i+j\leq  n/2-1\}
		\cup  \{(i, j; i + j) : n/2+1 \leq i+j;\ i,j\leq n-1\};$$
		$$C_2:=C_1\oplus (0,n/2); 
		C_3:=C_1\oplus (n/2,0); C_4:=C_1\oplus (n/2,n/2).$$ 
	\end{theorem}	 
	
	By Lemma \ref{DISJ}, 
	for any  $1\leq k\leq 4$,  $C_1\cup \dots\cup C_k$ is a $k$-strong defining set of $B_n$. 
	Now, define $I_{i,j}$ to be the intercalate in 
	$B_n$ on cells $(i,j)$, $(i+n/2,j)$, $(i,j+n/2)$ and  
	$(i+n/2,j+n/2)$. 
	Since a Latin trade must intersect any 
	critical set, 
	each intercalate of this form intersects each of $C_1$, $C_2$, $C_3$ and $C_4$ exactly once. 
	Moreover, 
	the set of intercalates 
	${\mathcal I}:=\{I_{i,j} \mid 0\leq i,j\leq n/2-1\}$
	partition $B_n$.   Thus, 
	a minimally $k$-strong defining set in $B_n$, for $n$ even cannot have size less than $kn^2/4$, for each $1\leq k\leq 4$. 
	This implies Theorem \ref{themaintwo}.

	\section{Tessellations of triangles and Latin trades in $B_n$}
	
	In what follows, 
	a triangle in the Euclidean plane is said to be {\em good} if it is right-angled with integer coordinates and a hypotenuse of gradient $-1$  (thus isoceles). 
	A tessellation of any finite region in the Euclidean plane into good triangles is also said to be {\em good}
	if no point is the vertex of more than $3$ of these triangles. 
	We also define 
	${\mathcal E}_n$ to be the good triangle in the Euclidean plane with vertices  
	$(0,0)$, $(0,n)$ and $(n,0)$.
	
	\begin{theorem} {\rm \cite{Dr91}} 
		Let $S$ be a good tessellation of  
		${\mathcal E}_n$ (with more than one triangle) and let $V$ be the set of coordinates of all vertices of the triangles in $S$.  
		Then $$T:=\{(i,j): (i,j)\in S\}\setminus \{(0,n),(n,0)\}$$ 
		is a Latin trade in $B_n$. 
		\label{tessa}
	\end{theorem}
	
	\begin{proof}
		This result is proved in \cite{Dr91}, but we add a short explanation here to be helpful. 
		Consider a triangle $\Delta$ in the tessellation with 
		vertex $(i,j)$ at the right angle. 
		Then $(i,j+k)$ and $(i+k,j)$ are the other vertices of 
		$\Delta$, where 
		$k$ is some non-zero integer. 
		We place $i+j+k\Mod{n}$ in the cell $(i,j)$ of the 
		disjoint mate of $T'$. 
		Since any point is the vertex of at most $3$ triangles in the tessellation, it follows that $(i,j)$ is not the point at the right-angle of any other triangle in the tessellation. Thus $T'$ is well-defined.  
		Moreover, since $k\neq 0$, $T$ and $T'$ are disjoint. 
		By similar reasoning, it can also be shown that 
		$T'$ and $T$ contain the same set of symbols in each row and column. 
	\end{proof}
	
	An example of the previous theorem can be seen in Figure \ref{trii}. Note that $(x,y)$-coordinates in the plane become visually transposed in the Latin square. That is, 
	the row and column are, visually, the $-y$ and $x$ axes, if we consider the origin to be the top-left corner of the Latin square.  This is a classic example of the confusion when switching between the Euclidean plane and an array.  
	The triangle in Figure \ref{trii} has thus been reflected on the $x$-axis so that the trade in $B_{11}$ can be better visualized. 
	The triangles in the tessellation of ${\mathcal E}_{11}$ 
	are given by:
	$$\begin{array}{l}
		\{(0,0),(0,4),(4,0)\}, \{(4,4),(0,4),(4,0)\},
		\{(0,4),(1,4),(0,5)\},  \{(0,4),(1,4),(1,5)\}, \\ 
		\{(0,5),(1,5),(0,6)\},  \{(0,5),(1,5),(1,6)\}, 
		\{(0,6),(1,6),(0,7)\},  \{(0,6),(1,6),(1,7)\}, \\
		\{(1,4),(4,4),(1,7)\},  \{(4,7),(4,4),(1,7)\},
		\{(0,7),(0,11),(4,7)\}, \{(4,0),(4,7),(11,0)\}.
	\end{array}$$ 
	
	\begin{figure}
		\begin{center}
			\begin{tikzpicture}[scale=0.6]
				\draw (0,0)--(0,11)--(11,11)--
				(0,0);
				\draw (0,7)--(7,7)--(7,11)--(6,10)--(6,11)--(5,10)--(5,11)--(4,10);
				\draw (0,7)--(4,11)--(4,7)--(7,10)--(4,10);
				
			\end{tikzpicture}
			\quad $\begin{array}{|c|c|c|c|c|c|c|c|c|c|c|}
				\hline
				{\bf 0_4} & 1 & 2 & 3 &  {\bf 4_5} & {\bf 5_6} & 
				{\bf 6_7} & {\bf 7_0} & 8 & 9 & 10 \\
				\hline
				1 &  2 & 3 &  4 & {\bf 5_8} & 
				{\bf 6_5} & {\bf 7_6} & {\bf 8_7} & 9 & 10 & 0 \\
				\hline
				2 & 3 & 4 & 5 & 
				6 & 7 & 8 & 9 & 10 & 0 & 1 \\
				\hline
				3 & 4 & 5 & 6 & 7 & 8 & 9 & 10 & 0 & 1 & 2\\
				\hline
				{\bf 4_0} & 5 & 
				6 & 7 & {\bf 8_4} & 9 &  10 & {\bf 0_8} & 1 & 2 & 3 \\
				\hline
				5 & 
				6 & 7 & 8 & 9 &  10 &  0 & 1 & 2 &  3 & 4 \\
				\hline
				6 & 7 & 8 & 9 & 10 & 0 & 1 & 2 & 3 & 4 & 5  \\
				\hline
				7 & 8 & 9 & 10 & 0 & 1 & 2 & 3 &  4 & 5 & 6  \\
				\hline
				8 &  9 & 10 & 0 & 1 & 2 & 3 & 4 &  5 &  6 & 7 \\
				\hline
				9 &  10 & 0 & 1 & 2 & 3 & 4 & 5 & 6 & 7 & 8 \\
				\hline
				10 &  0 & 1 & 2 & 3 & 4 & 5 & 6 &  7 &  8 & 9 \\
				\hline
			\end{array}$

			\caption{A good tessellation of ${\mathcal E}_{11}$  (reflected on the $x$-axis) and the corresponding Latin trade in $B_{11}$}
		\end{center} 
		\label{trii}
	\end{figure}

	The following lemma follows by recursively tessellating into maximal squares, which each tessellate into pairs of isosceles right-angled triangles.
	
	\begin{lemma}
		There is a good tessellation of any rectangle in the Euclidean plane with integer coordinates. 	
		\label{recta}
	\end{lemma}
	
	\begin{lemma}
		Let $2\leq 2m< n< 3m$.
		Then there is a Latin trade $T_{m,n}$ in $B_n$ such that:
		$$T_0:=\{(0,0;0),(m,0;m), (m,m;2m), (m,n-m;0)\}\subseteq T$$
		and if $(r,c;r+c)\in T_{m,n}\setminus T_0$, then
		$0\leq r\leq 3m-n$ and 
		$m\leq c\leq n-m$. 
		\label{doubletool}
	\end{lemma}
	
	\begin{proof}
		Consider the tessellation of 
		${\mathcal E}_n$ given by the triangles and rectangle  with the following vertex sets:
		$\{(0,0),(0,m),(m,0)\}$, 
		$\{(m,m),(0,m),(m,0)\}$, 
		$\{(m,0),(m,n-m),(n,0)\}$, 
		$\{(0,n-m),(m,n-m),(0,n)\}$, 
		$\{(3m-n,m),(3m-n,n-m),(m,m)\}$,
		$\{(m,n-m),(3m-n,n-m),(m,m)\}$,
		$\{(0,m),(0,n-m),
		(3m-n,m),(3m-n,n-m)\}$.
		The result then follows by Theorem \ref{tessa} and Lemma \ref{recta}.
	\end{proof}
	
	The trade in Figure \ref{trii} is in fact an example of the previous lemma with $n=11$ and $m=3$.

	\begin{figure}
		$$\begin{array}{|c|c|c|c|c|c|c|c|c|c|c|}
			\hline
			\mk 0 & \mk 1 & \mk 2 & \mk {\bf 3_7} & \mk 4 & 5 & 
			6 & {\bf 7_8} & {\bf 8_9} & {\bf 9_{10}} & {\bf 10_3} \\
			\hline
			1 & \mk 2 & \mk 3 & \mk 4 & \mk 5 & 
			\mk 6 & 7 & {\bf 8_0} & {\bf 9_8} & {\bf 10_9} & {\bf 0_{10}} \\
			\hline
			2 & 3 & \mk 4 & \mk 5 & 
			\mk 6 & \mk 7 & \mk 8 & 9 & 10 & 0 & 1 \\
			\hline
			3 & 4 & 5 & 
			\mk 6 & \mk 7 & \mk 8 & \mk 9 & \mk 10 & 0 & 1 & 2\\
			\hline
			4 & 5 & 
			6 & {\bf 7_3} & \mk 8 & \mk 9 & \mk 10 & \mk {\bf 0_7} & \mk 1 & 2 & {\bf 3_0} \\
			\hline
			5 & 
			6 & 7 & 8 & 9 & \mk 10 & \mk 0 & \mk 1 & \mk 2 & \mk 3 & 4 \\
			\hline
			6 & 7 & 8 & 9 & 10 & 0 & \mk 1 & \mk 2 & \mk 3 & \mk 4 & \mk 5  \\
			\hline
			\mk 7 & 8 & 9 & 10 & 0 & 1 & 2 & \mk 3 & \mk 4 & \mk 5 & \mk 6  \\
			\hline
			\mk 8 & \mk 9 & 10 & 0 & 1 & 2 & 3 & 4 & \mk 5 & \mk 6 & \mk 7 \\
			\hline
			\mk 9 & \mk 10 & \mk  0 & 1 & 2 & 3 & 4 & 5 & 6 & \mk 7 & \mk 8 \\
			\hline
			\mk 10 & \mk  0 & \mk 1 & \mk 2 & 3 & 4 & 5 & 6 &  7 &  8 & \mk 9 \\
			\hline
		\end{array}$$
		\caption{A Latin trade in $B_{11}$ that intersects $P_{11}$ twice.}
	\end{figure}

	\begin{lemma}
		Let $n> 3m\geq 3$.
		Then there is a Latin trade $T_{m,n}$ in $B_n$ such that:
		$$T_0:=\{(0,0;0),(m,0;m),
		(m,m;2m),(0,m;m)\}\subseteq T$$
		and if $(r,c;r+c)\in T_{m,n}\setminus T_0$, then
		$0\leq r\leq m$ and 
		$2m\leq c\leq n-m$. 
		\label{tripletool}
	\end{lemma}

	\begin{proof}
		Consider the tessellation of 
		${\mathcal E}_n$ given by the triangles and rectangle with the following vertex sets:
		$\{(0,0),(0,m),(m,0)\}$,
		$\{(m,m),(0,m),(m,0)\}$,
		$\{(0,m),(0,2m),(m,m)\}$, 
		$\{(m,2m),(0,2m),(m,m)\}$,
		$\{(0,n-m),(0,n),(m,n-m)\}$, 
		$\{(m,0),(m,n-m),(n,0)\}$, 
		$\{(0,2m),(0,n-m),
		(m,2m),(m,n-m)\}$.
		The result then follows by Theorem \ref{tessa} and Lemma \ref{recta}.
	\end{proof}

	\begin{figure}
		$$\begin{array}{|c|c|c|c|c|c|c|c|c|c|c|}
			\hline
			\mk 0 & \mk 1 & \mk {\bf 2_5} & \mk 3 & \mk 4 & {\bf 5_8} & 
			6 & 7 & {\bf 8_9} & {\bf 9_{10}} & {\bf 10_2} \\
			\hline
			1 & \mk 2 & \mk 3 & \mk 4 & \mk 5 & 
			\mk 6 & 7 & 8 & {\bf 9_0} & {\bf 10_9} & {\bf 0_{10}} \\
			\hline
			2 & 3 & \mk 4 & \mk 5 & 
			\mk 6 & \mk 7 & \mk 8 & 9 & 10 & 0 & 1 \\
			\hline
			3 & 4 & {\bf 5_2} & 
			\mk 6 & \mk 7 & \mk {\bf 8_5} & \mk 9 & \mk 10 & {\bf 0_8} & 1 & {\bf 2_0} \\
			\hline
			4 & 5 & 
			6 & 7 & \mk 8 & \mk 9 & \mk 10 & \mk 0 & \mk 1 & 2 & 3 \\
			\hline
			5 & 
			6 & 7 & 8 & 9 & \mk 10 & \mk 0 & \mk 1 & \mk 2 & \mk 3 & 4 \\
			\hline
			6 & 7 & 8 & 9 & 10 & 0 & \mk 1 & \mk 2 & \mk 3 & \mk 4 & \mk 5  \\
			\hline
			\mk 7 & 8 & 9 & 10 & 0 & 1 & 2 & \mk 3 & \mk 4 & \mk 5 & \mk 6  \\
			\hline
			\mk 8 & \mk 9 & 10 & 0 & 1 & 2 & 3 & 4 & \mk 5 & \mk 6 & \mk 7 \\
			\hline
			\mk 9 & \mk 10 & \mk  0 & 1 & 2 & 3 & 4 & 5 & 6 & \mk 7 & \mk 8 \\
			\hline
			\mk 10 & \mk  0 & \mk 1 & \mk 2 & 3 & 4 & 5 & 6 &  7 &  8 & \mk 9 \\
			\hline
		\end{array}$$
		\caption{Another Latin trade in $B_{11}$ that intersects $P_{11}$ twice.}
	\end{figure}

	We will proceed to prove Theorem \ref{themainone} via the use of Lemma \ref{intermed}. 
	
	\begin{lemma}
		Let $n\geq 2$. 
		For each element $(i,j;k)$ of $P_n$, there exists a Latin trade $T\subset B_n$ such that 
		$(i,j;k)\in T$ and $|T\cap B_n|=2$.
		\label{2exist}
	\end{lemma} 	
	
	\begin{proof}
		Since $P_n\oplus (1,1)=P_n$, without loss of generality, we may assume 
		that 
		$i=0$. 
		We first consider when $n$ is 
		even. Then observe that the Latin trade 
		$$T:=\{(0,j;j),(0,j+n/2,j+n/2),
		(n/2,j;j+n/2),(n/2,j+n/2;j)
		\}$$
		intersects the diagonal $D_{j}$ and
		$D_{j+n/2}$  twice each.  
		Since 
		$
		n/2-1 < j+n/2 \leq n-1$,  
		the diagonal $D_{j+n/2}$ does not intersect $P_n$. 
		Thus $T$ intersects $P_n$ exactly twice. 
		
		Otherwise $n$ is odd. Observe that $P_n^T\oplus ((n+3)/2,0)=P_n$. Thus we can assume that   
		$\lceil (n-3)/4\rceil \leq j\leq  
		(n-3)/2$.

		\noindent {\bf Case 1}:  
		$\lceil (n-3)/4\rceil \leq j < (n-3)/3$. 
		By Lemma \ref{tripletool}, 
		there is a Latin trade 
		$T_{j+1,n}\oplus (0,j)$ 
		in $B_n$ 
		which includes 
		$(0,j;j),(0,2j+1;2j+1),
		(j+1,j;2j+1),(j+1,2j+1,3j+2)$ 
		and all other elements in 
		cells $(r,c)$ where
		$0\leq r\leq j+1$ and 
		$3j+2\leq c\leq n-1$.  
		Observe that $T_{j+1,n}\oplus (0,j)$ intersects $D_{j}$ twice
		(at $(0,j;j)$ and $(j+1,j;2j+1)$), with all other elements in $D_{\alpha}$ for some 
		$\alpha$ such that 
		$$(n-3)/2< 2j+1\leq \alpha\leq n-1.$$  
		Thus this trade intersects 
		$P_n$ exactly twice.  
		
		\noindent {\bf Case 2}:  
		$j = (n-3)/3$.
		Then there is a Latin trade 
		$$T=
		\{(0,j;j),(0,2j+1;2j+1),
		(j+1,j;2j+1),(j+1,2j+1;3j+2),
		(0,3j+2;3j+2),(j+1,3j+2;j)\}.$$
		This intersects $D_j$, $D_{2j+1}$ and $D_{3j+2}=D_{n-1}$ twice each.
		Since $2j+1=(2n-3)/3>(n-3)/2$, 
		$D_{2j+1}$ does not intersect $P_n$.
		
		\noindent {\bf Case 3}:  
		$(n-3)/3 < j\leq (n-3)/2$.
		By Lemma \ref{doubletool}, 
		there is a Latin trade 
		$T_{j+1,n}\oplus (0,j)$ 
		in $B_n$ 
		which includes 
		$(0,j;j),
		(j+1,j;2j+1),(j+1,2j+1;3j+2),
		(j+1,n-1;j)$ 
		and all other elements in 
		cells $(r,c)$ where
		$0\leq r\leq 3j+3-n$ and 
		$2j+1\leq c\leq n-1$.   
		Observe that $T_{j+1,n}\oplus (0,j)$ intersects $D_{j}$ twice
		(at $(0,j;j)$ and $(j+1,j;2j+1)$),
		with all other elements in $D_{\alpha}$ for some 
		$\alpha$ such that 
		$$(n-3)/2< n-2-j\leq \alpha\leq n-1.$$  
		Thus this trade intersects 
		$P_n$ exactly twice.    
	\end{proof}	
	
	Next, as an intermediary step to proving Theorem \ref{themainone}, 
	we define $Q_n\subseteq P_n$ as follows: 
	$$Q_n=
	\{(i,j;i+j\Mod{n}): 
	1\leq i\leq j\leq 
	i+\lfloor (n-3)/2\rfloor \}.$$
	\begin{figure}
		$$\begin{array}{|c|c|c|c|c|c|c|c|c|c|c|}
			\hline
			0 & 1 & 2 & 3 &  4 & 5 & 
			6 & 7 & 8 & 9 & 10 \\
			\hline
			1 & \mk 2 & \mk 3 & \mk 4 & \mk 5 & 
			\mk 6 & 7 & 8 & 9 & 10 & 0 \\
			\hline
			2 & 3 & \mk 4 & \mk 5 & 
			\mk 6 & \mk 7 & \mk 8 & 9 & 10 & 0 & 1 \\
			\hline
			3 & 4 & 5 & 
			\mk 6 & \mk 7 & \mk 8 & \mk 9 & \mk 10 & 0 & 1 & 2\\
			\hline
			4 & 5 & 
			6 & 7 & \mk 8 & \mk 9 & \mk 10 & \mk 0 & \mk 1 & 2 & 3 \\
			\hline
			5 & 
			6 & 7 & 8 & 9 & \mk 10 & \mk 0 & \mk 1 & \mk 2 & \mk 3 & 4 \\
			\hline
			6 & 7 & 8 & 9 & 10 & 0 & \mk 1 & \mk 2 & \mk 3 & \mk 4 & \mk 5  \\
			\hline
			7 & 8 & 9 & 10 & 0 & 1 & 2 & \mk 3 & \mk 4 & \mk 5 & \mk 6  \\
			\hline
			8 &  9 & 10 & 0 & 1 & 2 & 3 & 4 & \mk 5 & \mk 6 & \mk 7 \\
			\hline
			9 &  10 & 0 & 1 & 2 & 3 & 4 & 5 & 6 & \mk 7 & \mk 8 \\
			\hline
			10 &  0 & 1 & 2 & 3 & 4 & 5 & 6 &  7 &  8 & \mk 9 \\
			\hline
		\end{array}$$
		\caption{The critical set $Q_{11}$}
	\end{figure}

	\begin{lemma}
		The {\rm PLS} 	
		$Q_n$ has unique completion to $B_n$. 	
	\end{lemma}
	
	\begin{proof}
		We give below a sequence of cells whose subsequent completion from $Q_n$ is ``forced'' to obtain $B_n$. 	By ``forced'', we mean that 
		we can sequence the empty cells so that when placing a symbol in cell $(i,j)$,   there is only one symbol which occurs in neither row $i$ nor column $j$ of the defining set so far. 
		In what follows, let $N=\lceil (n-3)/2\rceil$. 
		
		$$\begin{array}{l}
			(1,N+1),(2,N+2),\dots , (n-(N+1),n-1), \\
			(1,N+2),(2,N+3),\dots , (n-(N+2),n-1), \\  
			\dots, \\
			(1,n-1), \\
			(0,n-1),  \\
			(n-1,n-2),(0,n-2) \\
			(n-2,n-3), (n-1,n-3), (0,n-3)   \\
			\dots \\
			(1,0), (2,0),\dots ,(n-1,0),(0,0). \\ 
		\end{array}$$ 
	\end{proof}
	
	Since $P_n=P_n\oplus (1,1)$ and
	$Q_n\subset P_n$, we have the following corollary. 
	
	\begin{corollary}
		If $P\subset P_n$ such that $|P_n\setminus P|=1$, 
		then $P$ is a defining set of $B_n$.  	
	\end{corollary}
	
	The previous corollary, together with Lemmas \ref{2exist} and \ref{intermed}, imply Theorem 
	\ref{themainone}.
	As an extra, we next establish that $Q_n$ is a minimally $1$-strong defining set (i.e. a critical set) of $B_n$. 
	
	\begin{lemma}
		For each $(i,j;k)\in Q_n$, 
		there exists a Latin trade $T\subset B_n$ such  
		$T\cap Q_n=\{(i,j;k)\}$. 	
	\end{lemma}
	
	\begin{proof} 
		First observe that
		$$Q_n=\{(n-j,n-i;n-k):(i,j,k)\in Q_n\}.$$
		We can therefore assume, 
		without loss of generality, that 
		$i+j\leq n$.
		
		So we can consider just two cases: (1) $j<n/2$; and (2)
		$j\geq n/2$ and 
		$i+j\leq n$.		
		
		For Case (1) we use the trade $T\oplus (i,i-1)$, 
		where $T$ is the Latin trade based on the tessellation of ${\mathcal E}_n$ with triangles on vertices
		$\{(0,0),(0,j-i+1),
		(j-i+1,0)
		\}$, 
		$\{(j-i+1,j-i+1),(0,j-i+1),
		(j-i+1,0)
		\}$,
		$\{(0,j-i+1),(0,n),
		(n-j+i-1,j-i+1)
		\}$,
		$\{(n-j+i-1,0),(n,0),
		(n-j+i-1,j-i+1)
		\}$ and
		the rectangle on the set of vertices
		$\{
		(j-i+1,0),(j-i+1,j-i+1),
		(i-j-1,0)
		(j-i+1,n-j+i-1)
		\}$.
		
		For Case (2), 
		if $n$ is even, 
		use the Latin trade on cells $(i,j)$, $(i+n/2,j)$, 
		$(i,j+n/2)$ and $(i+n/2,j+n/2)$. 
		
		Otherwise $n$ is odd. Let $N=(n-1)/2$. We use the Latin trade on the
		set of cells 
		$$\{(i,j),(i,j-N)\}\cup 
		\{(i+N,c),(i+N+1,c): 
		j-N\leq c\leq j \}.$$
	\end{proof}
	
	\begin{corollary}
		The {\rm PLS} $Q_n$ is a critical set of $B_n$ for all $n\geq 2$. 
	\end{corollary}

	\section{Computational results}

	Although we have proven that $P_n$ is a minimally $2$-strong defining set for $B_n$, in general for odd $n$, it is not a minimal $2$-strong defining set of minimum size.
	For $n=5$, we have verified by computer that the minimum such size is given by $9$, an example of which can be seen in Figure \ref{min5}.
	\begin{figure}
		$$\begin{array}{|c|c|c|c|c|}
			\hline
			0 & 1 & & &  \\
			\hline
			1 &  & & 4 &  \\
			\hline
			& 3 & & & \\
			\hline
			& & 0 & & 2 \\
			\hline
			& & & 2 & 3 \\
			\hline
		\end{array}$$
		\caption{A $2$-strong defining set of minimum size $9$ in 
			$B_5$.} 
		\label{min5}
	\end{figure}
	
	
	\begin{table}[ht]
		\centering
		\renewcommand{\arraystretch}{1.2}
		\newcommand{\shaded}{\cellcolor{gray!25}\phantom{00}}
		\newcommand{\shadedval}[1]{\cellcolor{gray!25}#1}
		
		\begin{tabular}{|c|*{10}{c|}}
			\hline
			\diagbox{$k$}{$n$} & 2 & 3 & 4 & 5 & 6 & 7 & 8 & 9 & 10 & 11 \\
			\hline
			1  & 1 & 2 & 4 & 6 & 9 & 12 & 16 & 20 & 25 & \shadedval{22--30} \\
			\hline
			2  & 2 & 3 & 8 & 9 & 18 & \shadedval{18} & 32 & \shadedval{32} & 50 & \shadedval{32--45} \\
			\hline
			3  & 3 & 5 & 12 & 12 & 27 & \shadedval{21} & 48 & \shadedval{45} & 75 & \shadedval{44--56} \\
			\hline
			4  & 4 & 6 & 16 & 15 & 36 & \shadedval{27} & 64 & \shadedval{54} & 100 & \shadedval{55--67} \\
			\hline
			5  & \shaded & 8 & \shaded & 19 & \shaded & \shadedval{34} & \shaded & \shadedval{72} & \shaded & \shadedval{65--76} \\
			\hline
			6  & \shaded & 9 & \shaded & 20 & \shaded & \shadedval{39} & \shaded & 81 & \shaded & \shadedval{75--86} \\
			\hline
			7  & \shaded & \shaded & \shaded & 24 & \shaded & \shadedval{42} & \shaded & \shaded & \shaded & \shadedval{85--94} \\
			\hline
			8  & \shaded & \shaded & \shaded & 25 & \shaded & \shadedval{48} & \shaded & \shaded & \shaded & \shadedval{96--99} \\
			\hline
			9  & \shaded & \shaded & \shaded & \shaded & \shaded & 49 & \shaded & \shaded & \shaded & \shadedval{108} \\
			\hline
			10 & \shaded & \shaded & \shaded & \shaded & \shaded & \shaded & \shaded & \shaded & \shaded & \shadedval{110} \\
			\hline
			11 & \shaded & \shaded & \shaded & \shaded & \shaded & \shaded & \shaded & \shaded & \shaded & \shadedval{120} \\
			\hline
			12 & \shaded & \shaded & \shaded & \shaded & \shaded & \shaded & \shaded & \shaded & \shaded & 121 \\
			\hline
		\end{tabular}
		
		\caption{Minimum $k$-strong defining set sizes in 
			$B_n$.} 
		\label{min_k_bn}
	\end{table}
	
	Table~\ref{min_k_bn} shows sds($B_{n},k$) for $1 \leq k \leq 12$ and $1 \leq n \leq 11$.  These results were obtained by using the integer programming approach described in~\cite{Be05}.
	
	For $n \leq 5$, we were able to generate all trades in the Latin square $B_{n}$. Using these trades we constructed a 0-1 integer linear program (ILP) to find the size of the minimal set in $B_{n}$ that intersected every trade in at least $k$ places. For $n$ even, 
	the results agree with 
	Theorem \ref{themaintwo}. 
	For $n > 6$ we were unable to generate all possible trades in $B_{n}$ and thus some values, which are shaded, are simply lower bounds, or ranges in the case of $n=11$.
	
	One might conjecture from these results that if 
	$d_p$ is the size of the smallest Latin trade in $B_p$ (where $p$ is prime), then 
	removing any entry from $B_p$ results in a minimum $(d_p-1)$-strong defining set. This conjecture is equivalent to asserting that any pair of entries in $B_p$ is contained in a Latin trade of size $d_p$.

	\begin{table}[ht]
		\centering
		\renewcommand{\arraystretch}{1.2}
		\newcommand{\shaded}{\cellcolor{gray!25}\phantom{00}}
		
		\begin{tabular}{|c|*{12}{c|}}
			\hline
			\diagbox{$k$}{main class}
			& 6.1 & 6.2 & 6.3 & 6.4 & 6.5 & 6.6 & 6.7 & 6.8 & 6.9 & 6.10 & 6.11 & 6.12 \\
			\hline
			1 & 9  & 11 & 11 & 10 & 10 & 11 & 12 & 11 & 10 & 10 & 10 & 11 \\
			\hline
			2 & 18 & 18 & 16 & 14 & 14 & 16 & 20 & 18 & 18 & 16 & 14 & 14 \\
			\hline
			3 & 27 & 27 & 24 & 21 & 21 & 22 & 29 & 27 & 27 & 23 & 19 & 20 \\
			\hline
			4 & 36 & 36 & 36 & 28 & 28 & 32 & 36 & 36 & 36 & 28 & 24 & 24 \\
			\hline
			5 & \shaded & \shaded & \shaded & \shaded & \shaded & \shaded & \shaded & \shaded & \shaded & \shaded & \shaded & 32 \\
			\hline
			6 & \shaded & \shaded & \shaded & \shaded & \shaded & \shaded & \shaded & \shaded & \shaded & \shaded & \shaded & 36 \\
			\hline
		\end{tabular}
		\caption{Minimum $k$-strong defining set sizes in 
			main classes of order 6.} 
		\label{min_ls6}
	\end{table}
	
	Table~\ref{min_ls6} shows minimum $k$-strong defining set sizes for Latin squares in the main classes of order 6,
	using the main class numbering from~\cite{CRC}. Here, note that 6.12 is a main class with no intercalates ($2\times 2$ subsquares) and sds$(L,4)=36$ implies $L$ can be decomposed into 9 disjoint intercalates; otherwise sds$(L,4)<36$. The results for $k=1$ are the same as in~\cite{ABK03} using a different numbering.
	

	\section{Minimal  $k$-strong PLS in $B_n$ for larger $k$}
	
	For any $0\leq k < n$, define  
	$Q_{n,k}=D_0\cup D_1\cup \dots \cup D_k$.    
	Note that if $k=\lfloor (n-3)/2\rfloor$, $Q_{n,k}=Q_n$. 
	
	\begin{theorem}
		Let $x\leq n$ such that $n-4n/x\geq 30$. 
		there exists a Latin trade in $B_n$ which
		intersects the {\rm PLS} $Q_{n,n-\lceil 4n/x\rceil -30}$ at most $10\log_{4}(2x)$ times.  
		\label{trend}
	\end{theorem}
	
	\begin{proof}
		Define a sequence $m_0,m_1,\dots$ where
		$m_0=(n-3)/2$ and
		$m_i=\lfloor (m_{i-1}-3)/4 \rfloor$ for $i>0$. 
		Let $\alpha$ be the smallest integer such that 
		$m_{\alpha}\leq n/x+6$.   
		Then, recursively, we have that: 
		\begin{eqnarray}
			m_{\alpha} & \leq & \frac{m_0}{4^{\alpha}}
			= \frac{(n-3)}{2\times 4^{\alpha}}.
			\label{equationCC}
		\end{eqnarray} 
		
		Since $n-4n/x\geq 30$, 
		$n/x+6\leq (n-6)/4<m_0$. 
		Thus $\alpha\geq 1$ and $m_{\alpha-1}$  exists.
		By definition, $m_{\alpha-1}>n/x+6$. 
		Therefore:
		\begin{eqnarray}
			m_{\alpha} 
			=  
			\lfloor (m_{\alpha-1}-3)/4 \rfloor \geq (m_{\alpha-1}-6)/4
			& > & n/4x. 
			\label{equationCC2}
		\end{eqnarray}
		
		Next, tessellate ${\mathcal E}_n$ into 
		an $m_0\times (m_0+3)$ rectangle and two good triangles. 
		We will proceed to recursively tessellate (in a ``good'' way) 
		this rectangle into a number of good triangles and 
		increasingly small rectangles of dimension 
		$2m_i\times 2(m_i+3)$, for each $i\geq 0$. 
		This process is shown precisely in Figure 7, which is directly taken from \cite{Sz}.
		Note that at each stage we effectively tessellate an $m_i\times (m_i+3)$ rectangle,  magnifying this by a factor of $2$.  
		
		At each stage, from Figure 7, at most $10$ points of vertices outside the next rectangle are specified. 
		Thus the corresponding Latin trade will have at most $10\alpha$ cells outside the final $2m_{\alpha}\times 2(m_{\alpha}+3)$ rectangle, which by Lemma \ref{recta}, itself has a good tessellation.    
		
		Now, since $4m_{\alpha}+6\leq 4n/x+30$, by translation (using the operation $\oplus$) of the Latin trade constructed, we can position this final rectangle so that it does not intersect 
		$Q_{n,n-\lceil 4n/x\rceil -30}$.  
		Therefore the trade intersects $Q_{n,n-\lceil 4n/x\rceil-30}$ 
		at most
		$10\alpha$ times. But from (\ref{equationCC}) and (\ref{equationCC2}), 
		$$\alpha\leq \log_4{n}-\log_4{2m_{\alpha}}
		\leq \log_4{n} -\log_4{(n/(2x))} = \log_{4}(2x).$$   	
	\end{proof}
	
	\section{Conclusion}
	
	Lemma \ref{CHAIN} motivates the following open problem. If $P$ and $Q$ are, respectively, 
	minimally $k$-strong and minimally $k+1$-strong defining sets of a Latin square $L$ of order $n$, where $P\subset Q$, what is the possible size of $|P\setminus Q|$? 
	In this paper we have seen that the size in general can be quadratic in $n$, for example in Section 2, the size is $n^2/4$ for $1\leq k\leq 3$.  
	Also, in Section $3$, $|P_n\setminus Q_n|=(n^2-1)/8$ where $n$ is odd, $Q_n$ is a minimally $1$-strong defining set, 
	$P_n$ is a minimally $2$-strong defining set and $Q_n\subset P_n\subset B_n$.

	Finally, one might conjecture, given the results in this paper and the literature, that for odd $n$, $Q_{n,n-\lceil n/x\rceil }$ (defined in Section 5) is asymptotically minimally $\Theta(\log{x})$-strong.  
	From Theorem \ref{trend}, to prove such a conjecture, it would suffice to show that any Latin trade in $B_n$ intersects $Q_{n,n-\lceil n/x\rceil}$ at least $\Theta(\log{x})$ times. We leave this as an open problem.

	\begin{figure}
		\begin{center}
			\begin{tikzpicture}[
				box/.style={draw, thick},
				shade/.style={pattern=north east lines},
				font=\small
				]
				
				\begin{scope}
					\draw[box] (0,0) rectangle (6,4.5);
					\draw (3,1.5) -- (3,4.5);
					\draw[box] (0,1.5) -- (6,1.5);
					
					\draw[box,shade] (1.5,0) rectangle (6,1.5);
					
					\node at (1.5,2.7) {$2k+3$};
					\node at (4.5,2.7) {$2k+3$};
					\node at (0.75,0.75) {$2k$};
					\node at (4,0.75) {$2k\times2(k+3)$};
					
					\node[rotate=90] at (-0.35,2) {$4k+3$};
				\end{scope}
				
				\begin{scope}[xshift=8cm]
					\draw[box] (0,0) rectangle (5.5,4.5);
					\draw (0,2) -- (2.5,2);
					\draw (1.5,0) -- (1.5,2);
					\draw (2.5,1.5) -- (2.5,4.5);
					\draw (2,1.5) -- (2,2);
					
					\draw[box,shade] (1.5,0) rectangle (5.5,1.5);
					
					\node at (1.25,3.25) {$2k+3$};
					\node at (4,2.7) {$2k+4$};
					\node at (0.75,0.75) {$2k+1$};
					\node at (3.5,0.75) {$2k\times2(k+3)$};
					
					\node[rotate=90] at (-0.35,2) {$4k+4$};
				\end{scope}
				
				\begin{scope}[yshift=-6cm]
					\draw[box] (0,0) rectangle (6,4.5);
					\draw (0,2) -- (2.5,2);
					\draw (2,1.5) -- (2.5,1.5);
					\draw (2,0) -- (2,2);
					\draw (2.5,1) -- (2.5,4.5);
					

					\draw[box,shade] (2,0) rectangle (6,1);
					
					\node at (1.25,3.25) {$2k+3$};
					\node at (4.5,2.7) {$2k+5$};
					\node at (1,1) {$2k+2$};
					\node at (4,0.5) {$2k\times2(k+3)$};
					
					\node[rotate=90] at (-0.35,2) {$4k+5$};
				\end{scope}
				
				\begin{scope}[xshift=8cm,yshift=-6cm]
					\draw[box] (0,0) rectangle (5,4);
					\draw (2,0) -- (2,4);
					\draw (0,2) -- (2,2);
					
					\draw[box,shade] (2,0) rectangle (5,1);
					
					\node at (1,3) {$2k+3$};
					\node at (1,1) {$2k+3$};
					\node at (3.5,2.5) {$2k+6$};
					\node at (3.5,0.5) {$2k\times2(k+3)$};
					
					\node[rotate=90] at (-0.35,2) {$4k+6$};
				\end{scope}
				
			\end{tikzpicture}
			\caption{Figure 1 from \cite{Sz}}
		\end{center}
		\label{szzz2}
	\end{figure}

	\let\oldthebibliography=\thebibliography
	\let\endoldthebibliography=\endthebibliography
	\renewenvironment{thebibliography}[1]{%
		\begin{oldthebibliography}{#1}%
			\setlength{\parskip}{0.4ex plus 0.1ex minus 0.1ex}%
			\setlength{\itemsep}{0.4ex plus 0.1ex minus 0.1ex}%
		}%
		{%
		\end{oldthebibliography}%
	}
	

\end{document}